\documentclass[11pt, a4paper]{amsart}
\usepackage{amsmath}
\usepackage{amsthm}
\usepackage{amssymb}
\usepackage[sorting=nty,style=alphabetic, doi=false,isbn=false,url=false,eprint=false]{biblatex}
\usepackage[only,mapsfrom]{stmaryrd}
\usepackage{graphicx}
\usepackage{cancel} 
\usepackage{color} 
\usepackage[all]{xy} 
\usepackage{tikz-cd}
\usepackage{nicefrac}
\usepackage{enumitem}
\usepackage{mathrsfs}

\newtheorem{thm}{Theorem}[section]
\newtheorem{lem}[thm]{Lemma}
\newtheorem{prop}[thm]{Proposition}
\newtheorem{cor}[thm]{Corollary}

\newtheorem{remark}[thm]{Remark}

\newtheorem*{thm*}{Theorem}
\newtheorem*{conj*}{Conjecture}
\newtheorem*{cor*}{Corollary}
\newtheorem*{ques*}{Question}

\newtheorem*{namedthm}{\namedthmname}

\theoremstyle{definition}

\newtheorem{defn}[thm]{Definition} 
\newtheorem*{rem*}{Remark}

\DeclareMathOperator{\lk}{lk}
\DeclareMathOperator{\st}{st}

 \newcommand{\BN}{{\mathbb {N}}}
 
 \newcommand{\BR}{{\mathbb {R}}}
\newcommand{\BS}{{\mathbb {S}}}

 \newcommand{\BZ}{{\mathbb {Z}}}
 
\newcommand{\CC}{{\mathcal {C}}}

\newcommand{\CK}{{\mathcal {K}}}

\newcommand{\CS}{{\mathcal {S}}}

\newcommand{\FB}{{\mathfrak{B}}}

\newcommand{\FX}{{\mathfrak {X}}}

\newcommand{\id}{{{\rm id}}}

\newcommand{\Stein}{{\mathfrak{X}}}

\def\qed{\hfill$\square$\smallskip}

\DeclareMathOperator{\Map}{Map}
\DeclareMathOperator{\Homeo}{Homeo}

\DeclareMathOperator{\Ends}{Ends}

\DeclareMathOperator{\dlk}{{\lk}^{\downarrow}}
\DeclareMathOperator{\dst}{{\st}^{\downarrow}}

\DeclareMathOperator{\Sym}{Sym}

\begin{document}

\title{Block mapping class groups and their finiteness properties}

\author{J. Aramayona}
\address{Javier Aramayona: Instituto de Ciencias Matem\'aticas (ICMAT). Nicol\'as Cabrera, 13--15. 28049, Madrid, Spain}
\email{javier.aramayona@icmat.es}
\thanks{J.A. was supported by grant CEX2019-000904-S  funded by MCIN/AEI/ 10.13039/501100011033 and by grant PGC2018-101179-B-I00.  M.C. was supported by the Spanish grants P20\_01109 (financed by Junta de Andaluc\'ia) and PID2020-117971GB-C21 (financed by Ministerio de Ciencia e Innovación, Spain), the Research Program ``Braids" of ICERM (Providence, RI), and a Ram\'on y Cajal Grant 2021 (financed by Ministerio de Ciencia e Innovación, Spain) . R.S. was partially supported by NSF grant DMS--2005297. X.W. was partially supported by the Start-up Grant  of Fudan University.}

\author{J. Aroca}
\address{Julio Aroca: Instituto de Ciencias Matem\'aticas (ICMAT). Nicol\'as Cabrera, 13--15. 28049, Madrid, Spain}
\email{julio.aroca.93@gmail.com}
\author{M. Cumplido}
\address{Mar\'ia Cumplido: 
Departmento de \'Algebra,
Facultad de Matem\'aticas,
Universidad de Sevilla
Calle Tarfia s/n
41012, Seville, Spain}
\email{cumplido@us.es}
\author{R. Skipper}
\address{Rachel Skipper: Department of Mathemtics, The Ohio State University, 231 W 18th Ave, Columbus, OH 43210, USA}
\email{skipper.rachel.k@gmail.com}
\author{X. Wu}
\address{Xiaolei Wu: Shanghai Center for Mathematical Sciences, Jiangwan Campus, Fudan University, No.2005 Songhu Road,Shanghai, 200438 P.R. China}
\email{xiaoleiwu@fudan.edu.cn}

\subjclass[2010]{}

\date{\today}


\begin{abstract} 
Given $g \in \mathbb N \cup \{0, \infty\}$, let $\Sigma_g$ denote the closed surface of genus $g$ with a Cantor set removed, if $g<\infty$; or the blooming Cantor tree, when $g= \infty$. We construct a family $\mathfrak B(H)$ of subgroups of $\Map(\Sigma_g)$ whose elements preserve a {\em block decomposition} of $\Sigma_g$, and {\em eventually like act} like an element of $H$, where $H$ is a prescribed subgroup of the mapping class group of the block. The group $\mathfrak B(H)$ surjects onto an appropriate symmetric Thompson group of Farley--Hughes \cite{FH15}; in particular, it answers \cite[Question 5.37]{AV20} positively.

Our main result asserts that $\mathfrak B(H)$ is of type $F_n$ if and only if $H$ is. As a consequence, for every $g\in \mathbb N \cup \{0, \infty\}$ and every $n\ge 1$, we construct a subgroup $G <\Map(\Sigma_g)$ that is of type $F_n$ but not of type $F_{n+1}$, and which contains the mapping class group of every compact surface of genus $\le g$ and with non-empty boundary.
    \end{abstract}

\maketitle

\section{Introduction} 
Asymptotic mapping class groups of surfaces offer a fascinating bridge between mapping class groups of compact surfaces and Higman-Thompson groups \cite{Bri07, Deh06, FK04,FK08, FK09, AF21, GLU20, SW21b,ABF+21}. Originally introduced by Funar--Kapoudjian \cite{FK04,FK09} as discrete versions of the diffeomorphism group of the circle, they are strongly related to other classical families of groups that serve as ``artinizations'' of Higman-Thompson groups, such as the {\em  braided Thompson groups} of Brin \cite{Bri07} and Dehornoy \cite{Deh06}; see \cite{FKS13} for an account on this topic. 

A substantial piece of motivation for  asymptotic mapping class groups is that they are subgroups of mapping class groups of infinite-type surfaces, now commonly known as big mapping class groups. In this direction,  a  recent result of Funar--Neretin \cite{FN18} asserts that, the subgroup of the mapping class group of a closed surface minus a Cantor set whose elements can be smoothly extended to the closed surface coincides with a certain asymptotic mapping class group (the group of {\em half-twists} of \cite{FNg}). Moreover, by a result of Funar and the first-named author, this group is in fact dense in the  mapping class group of the underlying surface. 

Work of many authors shows that asymptotic mapping class groups --as well as some of their Thompson-like relatives-- tend to enjoy strong finiteness properties; more concretely, they are often of type $F_\infty$ \cite{BG84, FH15, FMW+13, BFM+16, GLU20, SW21a, ABF+21}.

In this paper we introduce a more general family of groups, which this time surject onto the {\em symmetric Higman-Thompson groups} of Farley--Hughes \cite{FH15} (more generally the Neretin groups \cite{Ne92}), thus providing an answer to \cite[Question 5.37]{AV20}. Moreover, we will prove that they realize all the possible types of finiteness properties. We now offer an informal definition of our groups sufficient to understand our main result, and postpone all details until Section \ref{section-can-sur}.

Let~$O$ and $Y$ be compact, connected, orientable surfaces, where $Y$ is further assumed to be closed. For each $d\ge 2$, write $Y^d$ for the surface that results from removing $d+1$ open disks from $Y$. 
From this data, we construct a non-compact surface $\CC_d(O,Y)$ by gluing copies of $Y^d$ to $O$ in a tree-like manner; see Figure \ref{fig:Cantorsurf}. Following \cite{ABF+21}, the surface $\CC_d(O,Y)$ is called the {\em Cantor surface} determined by $d$, $O$ and $Y$. Each of standard copies of~$Y^d$ used to construct $\CC_d(O,Y)$ is called a {\em block}. We will be interested in the group $\mathfrak B_d(H)$ whose elements are isotopy classes of homeomorphisms of $\CC_d(O,Y)$ which {\em eventually} send blocks to blocks and, on each block in a connected component, they {\em act like} an element of a fixed subgroup $H$ of the mapping class group of the block; see Section \ref{section-can-sur} for details. In the case when $H$ is trivial, the group $\mathfrak B_d(H)$ is the {\em asymptotic mapping class group} of \cite{ABF+21}. 

As we will see in Proposition \ref{prop:surj-symthom}, there is a  surjective homeomorphism from $\mathfrak B_d(H)$ onto an appropriate {\em symmetric Higman-Thompson group}, which simply records the action of  elements of $\mathfrak B_d(H)$ permute the ends of $\CC_d(O,Y)$. As a consequence, this answers Question 5.37 of \cite{AV20}. 

As we explain next, the group $\mathfrak B_d(H)$ has the same finiteness properties as $H$. Recall that a group $G$ is of \emph{type $F_n$} if there exists an aspherical CW-complex whose fundamental group is $G$ and whose $n$-skeleton is finite; a group is of \emph{type $F_\infty$} if it is of type $F_n$ for all $n\geq 1$. Our main result is: 

\begin{thm}
 Let $O$ be a compact orientable surface, and $Y$ a closed surface of genus $\le 1$. Then $\mathfrak B_d(H)$ is of type $F_n$ if and only if $H$ is. 
\label{thm:main} 
\end{thm} 

\begin{remark}
A minor variation of our construction produces a slightly larger family $\mathfrak B_{d,r} $ of groups, which surject onto the more general  symmetric Thompson groups $V_{d,r}(G)$; see Remark \ref{rmk:moreroots} for details. Theorem \ref{thm:main} above is still true in this marginally broader setting, by a word-for-word copy of our arguments below. 
\end{remark} 

\smallskip

\noindent{\bf Applications.} We now give an application of the above theorem. Write~$\Sigma_g$, with $0\le g \le \infty$, for the closed surface of genus $g\ge 0$ with a Cantor set removed (for $g<\infty$) or the {\em blooming Cantor tree} (for $g=\infty$); see Section \ref{section-can-sur}. By the classification of infinite-type surfaces \cite{Ke23,Ri63}, we know that $\CC_d(O,Y)$ is homeomorphic to $\Sigma_g$, for some $g$; see Remark \ref{rem:homeotype}.

When $Y$ is the sphere, $\Map(Y^d)$ contains the pure braid group with $d$ strings. By a theorem of Zaremsky \cite{Za17}, $\Map(Y^d)$ contains a subgroup of type $F_n$ but not of type $F_{n+1}$ for $d\geq n+3$. By choosing an appropriate subsurface, the same is true when $Y$ is a torus. Selecting these subgroups as local group $H$, we obtain: 

\begin{cor}
For every $n \ge 1$ and every $g\ge \mathbb N \cup \{\infty\}$, $\Map(\Sigma_g)$ contains a subgroup $G$ of type $F_n$ but not $F_{n+1}$ such that $G$ contains the mapping class group of every compact connected surface of genus $\le g$ with non empty boundary.  
\end{cor}

We remark that the above corollary also follows from the combination of the main result of \cite{ABF+21} with the {\em diagonal labeling} of our construction; see Remark \ref{rem:diag-bmap}. As both families of groups contain the asymptotic mapping class groups, they are dense inside $\Map(\Sigma_g)$  when $g<\infty$ \cite[Theorem~3.19]{SW21b}.

\medskip

\noindent{\bf Acknowledgements.} We would like to thank Louis Funar for conversations. We are grateful to Chris Leininger for an argument that fixes an error in Proposition 6.6 in a previous version; see Remark \ref{rem:chris}.

\section{Preliminaries}\label{section-pre}
In this section we provide the background material needed in our arguments. We refer the reader to \cite{FM12} and \cite{AV20}, respectively, for details on mapping class groups of finite and infinite type surfaces.

\subsection{Surfaces} Throughout, all surfaces will be connected, orientable, and second countable. If a surface $S$ has non empty boundary, we will further assume that $\partial S$ has finitely many connected components. We say that $S$ has {\em finite type} if $\pi_1(S)$ is finitely generated; otherwise it has  {\em infinite type}. An {\em end} of $S$ is an element of \[\varprojlim \pi_0( S\setminus K),\] where the inverse limit is taken over all compact sets $K \subset S$, directed with respect to inclusion. We say that an end is {\em planar} if it has a neighborhood of genus zero; otherwise it is {\em non-planar}. Denote the set of ends by  $\Ends(S)$ and its subset of non-planar ends by $\Ends_{np}(S)$.

The set of ends $\Ends(S)$ becomes a topological space by endowing it with the limit topology arising from the discrete topology on each of the terms defining the inverse limit above. Equipped with this topology, $\Ends(S)$ is homeomorphic to a closed subspace of the Cantor set, and $\Ends_{np}(S)$ is a closed subset of $\Ends(S)$.  By work of  Ker\'ekj\'art\'o \cite{Ke23} and Richards \cite{Ri63}, two surfaces are homeomorphic if and only if they have the same genus and number of boundary components, and there is a homeomorphism between their space of ends that restricts to a homeomorphism between the subspaces of non-planar ends; see \cite{Ri63} for a precise statement. 

For $g \in \mathbb N \cup \{0\}$, we denote by $\Sigma_g$ the connected orientable surface without boundary, of genus $g$ and with a Cantor space of ends, all of them planar; we remark that $\Sigma_0$ is usually referred to as the {\em Cantor tree surface}. Similarly, let $\Sigma_\infty$ be the {\em blooming Cantor tree surface}, namely the connected orientable  surface with empty boundary, infinite genus, and whose space of ends is a Cantor set and every end is non-planar. Finally, for $0\le g\le \infty$ and $b\ge 1$ write $\Sigma_g^b$ for the surface that results by removing $b$ open discs from $\Sigma_g$. 

\subsection{Curves} By a {\em curve} on $S$ we mean the isotopy class of a simple closed curve on $S$. We say that a curve on $S$ is {\em essential} if it does not bound a disc or an annulus whose other boundary component is a boundary component of $S$.

\subsection{Mapping class group} Given a surface $S$, write $\Homeo^+(S,\partial S)$ for the group of orientation preserving self-homeomorphisms of $S$ that restrict to the identity on $\partial S$. The {\em mapping class group} of $S$ is 
\[
\Map(S) := \pi_0\left(\Homeo^+(S,\partial S)\right).
\]
We will often blur the distinction between a homeomorphism and its isotopy class; however, if we want to make this distinction explicit, we will use $[f]$ to denote the isotopy class of the homeomorphism $f$. 

The mapping class group $\Map(S)$ is countable if and only if $S$ has finite type. When $S$ has infinite type, we endow $\Map(S)$ with the quotient topology of the compact-open topology on $\Homeo^+(S,\partial S)$; with respect to this topology,  $\Map(S)$ is a zero-dimensional, non-locally compact group; see \cite{AV20}.

\medskip

We will also need a version of the mapping class group that allows boundary components to be permuted. To this end, for each boundary circle $a \subset \partial S$ we fix a homeomorphism $p_a:a \to \mathbb S^1$, which we call a {\em parametrization} of $a$. We say that an orientation preserving homeomorphism $f:S \to S$ {\em respects the boundary parametrization} if $p_{f(a)} \circ f_{|a} \equiv p_a$ for all $a\in \pi_0(\partial S)$. We now set: 

\begin{defn}[Boundary-permuting mapping class group] 
The {\em boundary permuting mapping class group} $\Map_\partial(S)$ is the group of isotopy classes of orientation-preserving self-homeomorphisms of $S$ that respect the boundary parametrization, modulo isotopies that fix $\partial S$ pointwise. 
\label{def:spherepermuting}
\end{defn}

\begin{remark}\label{rem:finindex}
Since we are assuming that surfaces have a finite number of boundary components, $\Map(S)$ is a (normal) finite-index subgroup of $\Map_\partial(S)$; moreover, the two groups coincide if and only if $S$ has at most one boundary component. 
\end{remark}

Finally, given a surface $S$ and $A\subseteq \pi_0(\partial S)$, we define $\Map_\partial(S,A)$ to be the subgroup of $\Map_\partial(S)$ whose elements restrict to the identity on every element of $A$. Of course, if $A=\partial S$, then $\Map_\partial(S,A)= \Map(S)$.

\section{Cantor surfaces and block mapping class groups}\label{section-can-sur}

We now recall the definition of a {\em Cantor surface} from \cite{ABF+21}. Let~$O$ be an arbitrary compact orientable surface, and $Y$ homeomorphic to a sphere or a torus. Denote by~$Y^d$ the surface obtained from~$Y$ by removing $d+1$~open disks. Write $b_0, \ldots, b_d$ for the boundary components of $Y^d$; we will refer to~$b_0$ as the {\em top} boundary component of $Y^d$. For each $i= 1, \ldots, d$, we choose an orientation-reversing homeomorphism $\mu_i: b_0 \to b_i$. We inductively define a family~$\{O_n\}$ of compact surfaces  as follows: 
\begin{enumerate} 
\item $O_1$ is the surface that results by removing an open disk from $O$. 

\item $O_2$ is the surface obtained by gluing (with respect to some pre-fixed orientation-reversing homeomorphism) a copy of $Y^d$, along its top boundary component, to the boundary of~$O_1$ created in the previous step; 

\item For each $n\ge 2$,  $O_n$ is the surface obtained by gluing a copy of~$Y^d$, along its top boundary component, to each of the boundary components of $O_{n-1} \setminus O_{n-2}$, using the homeomorphisms $\mu_i$ above.  
\end{enumerate}

The {\em Cantor surface} $\CC_d(O,Y)$ is the union of the surfaces $O_k$ above. The closure of each of the connected components of $O_n \setminus O_{n-1}$ is called a {\em block} of $\CC_d(O,Y)$; by construction, every block is homeomorphic to $Y^d$. The {\em top boundary} $\partial_T B$ of a block $B$ is the image of the top boundary of $Y^d$ under the gluing map.

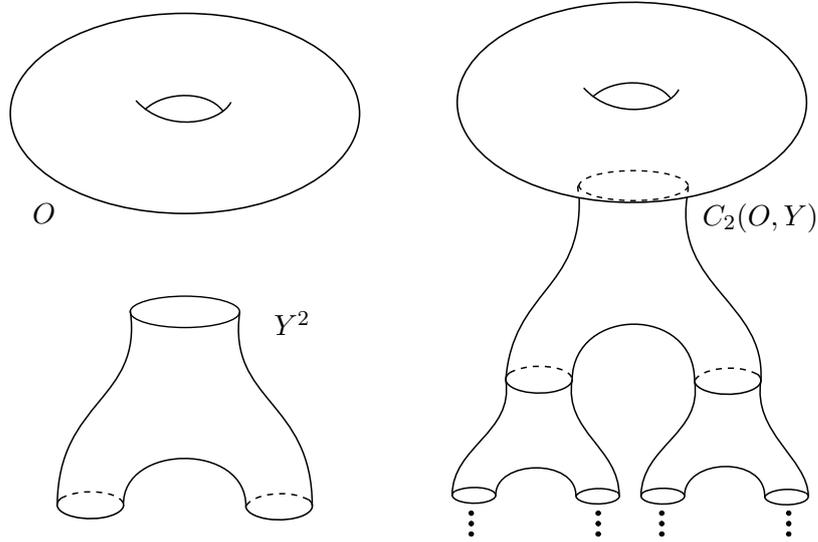
\begin{figure}[ht]
\centering
\begingroup%
  \makeatletter%
  \providecommand\color[2][]{%
    \errmessage{(Inkscape) Color is used for the text in Inkscape, but the package 'color.sty' is not loaded}%
    \renewcommand\color[2][]{}%
  }%
  \providecommand\transparent[1]{%
    \errmessage{(Inkscape) Transparency is used (non-zero) for the text in Inkscape, but the package 'transparent.sty' is not loaded}%
    \renewcommand\transparent[1]{}%
  }%
  \providecommand\rotatebox[2]{#2}%
  \newcommand*\fsize{\dimexpr\f@size pt\relax}%
  \newcommand*\lineheight[1]{\fontsize{\fsize}{#1\fsize}\selectfont}%
  \ifx\svgwidth\undefined%
    \setlength{\unitlength}{396.6414372bp}%
    \ifx\svgscale\undefined%
      \relax%
    \else%
      \setlength{\unitlength}{\unitlength * \real{\svgscale}}%
    \fi%
  \else%
    \setlength{\unitlength}{\svgwidth}%
  \fi%
  \global\let\svgwidth\undefined%
  \global\let\svgscale\undefined%
  \makeatother%
  \begin{picture}(1,0.50975212)%
    \lineheight{1}%
    \setlength\tabcolsep{0pt}%
    \put(0,0){\includegraphics[width=\unitlength,page=1]{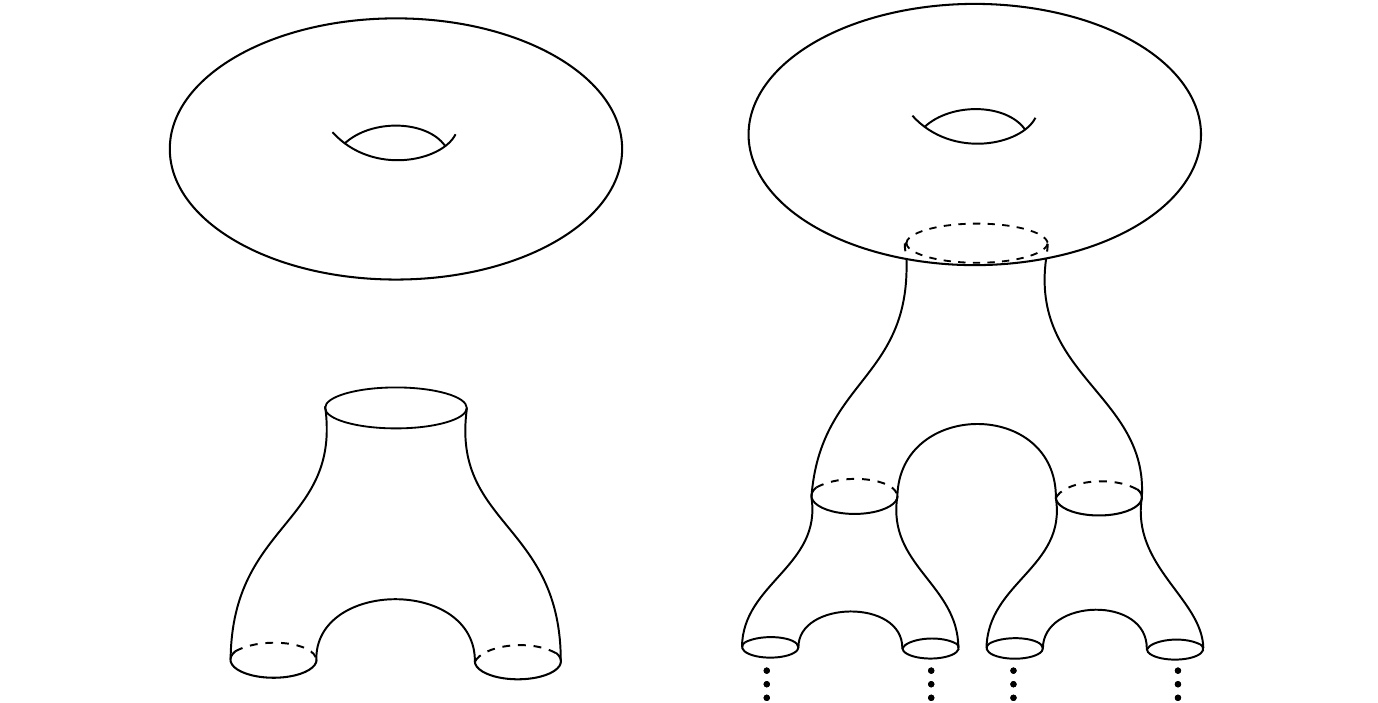}}%
    \put(0.14427667,0.29790618){\makebox(0,0)[lt]{\lineheight{1.25}\smash{\begin{tabular}[t]{l}$O$\end{tabular}}}}%
    \put(0.37127335,0.19152493){\makebox(0,0)[lt]{\lineheight{1.25}\smash{\begin{tabular}[t]{l}$Y^2$\end{tabular}}}}%
    \put(0.77431951,0.29601974){\makebox(0,0)[lt]{\lineheight{1.25}\smash{\begin{tabular}[t]{l}$C_2(O,Y)$\end{tabular}}}}%
  \end{picture}%
\endgroup%

\caption{Example of a Cantor surface}
\label{fig:Cantorsurf}
\end{figure}

\begin{remark}\label{rmk:moreroots} 
Following \cite{SW21b} and \cite{ABF+21}, one may slightly generalize the above construction by first removing $r\ge 1$ open discs from $O$ and then gluing copies of $Y^d$ to each of these boundary components. Our main result has a direct translation to the case of the surfaces $\CC_{d,r}(O,Y)$ obtained in this way; however, we have decided to restrict our attention to the case $r=1$ in order to minimize the notation and level of technicality. 
\end{remark}

A {\em suited subsurface} of $\CC_d(O,Y)$ is a connected subsurface that is the union of $O_1$ and finitely many blocks; in particular, a suited subsurface is always compact. A {\em suited curve} of $\CC_d(O,Y)$ is a boundary component of a suited subsurface that is not one of the original boundary components of $O$.  Given a suited subsurface $Z$, we refer to the boundary curves of $Z$ that are suited curves of~$\CC_d(O,Y)$ as the {\em suited boundary components} of $Z$; the rest (if any) are called the {\em primary boundary components} of $Z$. Given a suited boundary component $a \subset \partial Z$, we say that a connected component $X \subset \CC_d(O,Y)\setminus Z$ is {\em adjacent to} $a$ if the intersection of the closures of $Z$ and $X$ is equal to $a$.

\begin{remark}
The space of ends of $\CC_d(O,Y)$ is always homeomorphic to a Cantor set, and every two ends of $\CC_d(O,Y)$ are related by a homeomorphism of $\CC_d(O,Y)$. In particular, by the classification of non-compact surfaces of Ker\'ekj\'art\'o-Richards \cite{Ri63}, one has: 
\begin{itemize}
    \item If $Y\cong \BS^2$, then $\CC_d(O,Y)\cong \Sigma_g^b$, where $g$ and $b$ are, respectively,  the genus and number of boundary components of $O$. 
    \item If $Y\cong \BS^1 \times \BS^1$, then $\CC_d(O,Y)\cong \Sigma_\infty^b$, where $b$ is the number of boundary components of $O$. 
\end{itemize}
\label{rem:homeotype} 
\end{remark}

For the sake of notational convenience, for the rest of the section we will denote $\CC_d(O,Y)$ simply by~$\CC_d$, unless we need to make explicit reference to the surfaces~$O$ and~$Y$.

\begin{defn}[Block homeomorphism]
Let $\CC_d$ be a Cantor surface. We say that a homeomorphism $f: \CC_d \to \CC_d$ is {\em block} if there exists a suited subsurface $Z \subset \CC_d$ such that: 
\begin{enumerate}
    \item $f(Z)$ is also suited;
    \item The restriction $f: \CC_d \setminus Z \to \CC_d\setminus f(Z)$ sends blocks to blocks. 
\end{enumerate}
\end{defn}
The suited subsurface~$Z$ in the definition above is called a {\em defining surface} for~$f$. Note that  any surface containing a defining surface for~$f$ is itself a defining surface for~$f$. As an immediate consequence, the composition of two block homeomorphisms is also block; in particular, the set of block homeomorphisms of $\CC_d$ is in fact a group under composition.

Next, for each block $B \subset \CC_d$ we fix, once and for all, a homeomorphism $\iota_B:B \to Y^d$ that takes the top boundary component of $B$, denoted $\partial_T B$, to $b_0$, namely the top boundary component of $Y^d$. The homeomorphism $\iota_B$ induces an isomorphism (abusing notation) \[\iota_B: \Map_\partial(B,\partial_T B) \to \Map_\partial(Y^d, b_0),\] and  we will use $\iota_B$ to identify $\Map_\partial(B,\partial_T B)$ with  $\Map_\partial(Y^d, b_0)$ without further mention. Note that $\iota_B$ induces an order on the remaining boundaries of $B$ from that of $Y^d$.

\begin{defn}[$H$-block homeomorphism]
Let $H<\Map_\partial(Y^d, b_0)$, and  $f:\CC_d \to \CC_d$ be a block homeomorphism with defining surface $Z$. We say that $f$ is $H$-{\em block} if for every suited boundary component $a \subset \partial Z$ there is an element $h_a\in H$ such that, for every block $B$ in the connected component of $\CC_d\setminus Z$ adjacent to $a$, we have that 
\[[\iota_{f(B)}\circ f\circ \iota_B^{-1}] \equiv h_a.\] We will say that the boundary curve $a$ is {\em labeled by the element} $h_a$,  that $f$ is \emph{labeled by the tuple} $\mathbf h= (h_1, \ldots, h_n)$ outside of $Z$ (where $n$ is the number of suited boundary components of $Z$), and that $f$ {\em eventually acts like} $\mathbf h$.
\end{defn}

\begin{defn}[$H$-block mapping class group]
Let $\CC_d$ be a Cantor surface, and $H<\Map_\partial(Y^d,b_0)$. The {\em block mapping class group with local group}  $H$, denoted $\mathfrak B_d(H)$, is the group of isotopy classes of $H$-block homeomorphisms of $\CC_d$.
\end{defn}

In the case when $H=\{1\}$, we will simply write $\mathfrak B_d=\mathfrak B_d(H)$. We remark that this case corresponds to the asymptotic mapping class groups of \cite{ABF+21} which,  as mentioned in the introduction,  are in turn an abstraction of the ones considered by Funar, Kapoudjian and the first named author \cite{FK04, FK09, AF21}.

\begin{remark}
The group $\mathfrak B_d(H)$ is always countable, as $\Map(Y,b_0)$ is countable and so is the set of suited subsurfaces of $\CC_d(O,Y)$. In stark contrast, the group of isotopy classes of block homeomorphisms of $\CC_d(O,Y)$ is uncountable.
\end{remark}

\begin{remark}
Considering the surfaces $\CC_{d,r}(O,Y)$ as in Remark \ref{rmk:moreroots} would give rise to more general block mapping class groups $\mathfrak B_{d,r}(O,Y)$. As will become apparent, the obvious adaptation of our arguments below would yield the analog of Theorem \ref{thm:main} for this (a priori, larger) set of groups. 
\label{rem:morerootsgroups} 
\end{remark}

\subsection{An alternate description} 
We now give a useful alternate representation of elements of $\mathfrak B_d(H)$, reminiscent of the {\em tree-pair} representation of elements of (symmetric) Higman-Thompson groups; see Section \ref{sec:HigmanThompson}. 

Let $f \in \mathfrak B_d(H)$, and  consider a defining subsurface $Z \subset \CC_d$ for $f$, with suited boundary components $a_i$, $i=1, \ldots, n$. Write $X_i$ for the connected component of $\CC_d \setminus Z$ adjacent to $a_i$, and denote by $f_Z$ (resp. $h_i$) the restriction of $f$ to $Z$ (resp. the element of $H$ that $f$ acts like on $X_i$). With this notation, $f$ is determined by the tuple 
\[
\left(Z, (f_Z,\mathbf h),f(Z)\right),
\] 
where $\mathbf h= (h_1, \ldots, h_n)$. Finally, observe that, since defining surfaces are not unique, there are countably many different such tuples representing the same element of $\mathfrak B_d(H)$. We use $[Z, (f_Z,\mathbf h),f(Z)]$ to denote the equivalence class (with respect to the obvious equivalence relation) containing $\left(Z, (f_Z,\mathbf h),f(Z)\right)$.

\begin{remark}\label{rem:diag-bmap}
If one takes the labeling to be the same for all components, i.e. $\mathbf h = (h,\cdots,h)$ for some $h\in H$, then the set of all such elements forms a subgroup $\mathfrak{DB}_d(H)$. Since  the surjection $\mathfrak {DB}_d(H) \to H$ mapping  $[Z, (f_Z,\mathbf h),f(Z)]$ to $h$ splits, we have an  isomorphism $$\mathfrak{DB}_d(H) \cong \mathfrak B_d(\{1\})\rtimes H.$$ Now $\mathfrak B_d(\{1\})$ is of type $F_\infty$ \cite{ABF+21}, hence  $\mathfrak B_d(\{1\})\rtimes H$ if of type $F_n$ if and only if $H$ is \cite[Section 7.2]{Ge08}.
\end{remark}

\section{The relation with symmetric Higman-Thompson groups }
\label{sec:HigmanThompson}
In this section we will relate our groups to the so-called {\em symmetric Higman-Thompson groups}, introduced by Farley and Hughes in \cite{FH15}. We commence by reviewing the construction of symmetric Higman-Thompson groups; see \cite{FH15} for a detailed treatment. 

Let $\mathcal T_d$ be the infinite rooted tree of degree $d\ge 2$, where every vertex other than the root has degree~$d+1$. We endow~$\mathcal T_d$ with the path-metric that arises by deeming each edge to have length one; with respect to this path-metric, the end space of~$\mathcal T_d$ is naturally homeomorphic with the set of geodesics emanating from the root, which in turn is  homeomorphic to the standard Cantor set~$\mathfrak{C}$.

Given a vertex $v \in \mathcal T_{d}$ at distance $k\ge 0$ from the root, we say that the vertex~$w$ is a {\em descendant} of~$v$ if $d(v,w)=1$ and~$w$ is at distance~$k+1$ from the root. A  subtree~$F$ of $\mathcal T_{d}$ is {\em complete} if, for every vertex $v\in F$, whenever~$F$ contains a descendant of~$v$,  it contains all of its descendants. Observe that if $F$ is a complete finite rooted subtree of $\mathcal T_d$ then, for every $c\in \mathfrak C$ there is a unique leaf~$l$ of~$F$ such that the geodesic in~$\mathcal T_d$ representing~$c$ contains~$l$; we say that $l$ is the  leaf of $F$ {\em nearest to} ~$c$.

\subsection{Higman-Thompson groups} Let $F_-$ and $F_+$ be complete finite rooted subtrees of $\mathcal T_{d}$, with the same number of leaves, and $b$ a bijection between their sets of leaves. Then the triple $(F_-, b, F_+)$ extends to a homeomorphism $g: \mathfrak C \to \mathfrak C$ by identifying (in the obvious way) the infinite subtree of $\mathcal{T}_{d}$ below the leaf $l$ with that below the leaf $b(l)$. We say that the triple $(F_-, b, F+)$ is a {\em tree-pair} representative of $g$; of course, there are infinitely many tree-pair representatives of any given homeomorphism. 

A pleasant exercise shows that if two self-homeomorphisms of $\mathfrak{C}$ admit a tree-pair representative, then the same is true about their composition. Thus, we define the {\em Higman-Thompson group} $V_{d}$ \cite{Hi74} as the group of homeomorphisms of $\mathfrak{C}$ that have a tree-pair representative; we remark that~$V_{2}$ is the classical Thompson group $V$. A well-known fact \cite{Br87}  is that $V_d$ is of type~$F_\infty$.

\subsection{Symmetric Thompson groups} Next, we recall the definition of the {\em symmetric} versions of the Higman-Thompson groups, introduced by Farley and Hughes \cite{FH15}.  

Enumerate the descendants of every vertex with $\{0,\dots, d-1\}$. Since the (Cantor) space of ends of $\mathcal{T}_{d}$ can be identify with the set of geodesic of $\mathcal{T}_{d}$, we can also identify it with the set of infinite sequences $(t_1, t_2, t_3,\ldots)$, where  $0\le t_i< d$ for every $i\ge 1$. 

Let $F_-$ and~$F_+$ be complete finite subtrees of~$\mathcal T_{d}$, with the same number of leaves, and~$b$ a bijection between their sets of leaves. Write~$l_i$, $i=1, \ldots, n$ for the leaves of~$F_-$.

Let~$G$ be a subgroup of the symmetric group~$\Sym(d)$ on $d$~elements. With respect to the above notation, for every leaf $l_i$ we chose $g_i\in G$, and write $\mathbf g = \{g_1, \ldots, g_n\}$. 
Then the tuple 
\[
(F_-, (b,\mathbf g), F_+)
\] 
induces a homeomorphism $\psi:\mathfrak{C}\to \mathfrak{C}$ as follows. Let $c\in \mathfrak C$ be given in terms of an infinite sequence $(c_1,c_2,\dots)$, and consider the leaf~$l_i$ of~$F_-$ nearest to~$c$, so that $(c_1,c_2,\dots, c_m)$ represents the leaf~$l_i$. Then the leaf~$b(l_i)$ is represented by a finite sequence~$b(l_i)$ (using a slight abuse of notation), and define
\[
\psi(c) = (b(l_i), (g_i(c_{m+1}),g_i(c_{m+2}),\dots).
\]
 We say that $\psi$ {\em eventually acts like} $\mathbf g$, and the tuple $(F_-, (b,\mathbf g), F_+)$ is called the {\em symmetric tree-pair representative} of~$\psi$; as before, there are infinitely many symmetric tree-pair representatives of the same homeomorphism of~$\mathfrak{C}$. 

Again, if two self-homeomorphisms of $\mathfrak{C}$ admit a tree-pair representative, then the same is true about their composition. 
Hence we define the {\em symmetric Higman-Thompson group} $V_d(G)$ {\em with local group} $G$ as the group of homeomorphisms of $\mathfrak{C}$ that admit a $G$ tree-pair representative. 
Observe that if $G=\{1\}$ then $V_{d}(G)= V_{d}$.

Finally, we note that the groups $V_{d}(G)$ are examples of {\em FSS groups} in the sense of Farley and Hughes, which are all of type $F_\infty$ \cite{FH15}. 

\begin{remark}
The groups $V_d(G)$ are part of a larger family of symmetric Higman-Thompson groups $V_{d,r}(G)$ groups, where one allows $r$-forests of $d$-valent rooted trees instead of just $d$-valent rooted trees. These groups are also FFS groups in the sense of Farley and Hughes \cite{FH15}; as such, they are of type $F_\infty$.
\label{rem:thompson}
\end{remark}

\subsection{A surjective homomorphism} 
Recall from Section \ref{section-can-sur} that the Cantor surface $\CC_d(O,Y)$ is the union of compact surfaces $O_k$, with $k\ge 1$. 

We construct an abstract simplicial graph $\Gamma$ as follows: the vertices of $\Gamma$ are the blocks of $\CC_d(O,Y)$ plus $O_1$, and two vertices are adjacent in $\Gamma$ if and only if the corresponding (suited) subsurfaces of $\CC_d(O,Y)$ share one suited boundary component. Observe that, by construction, $\Gamma$ is an infinite rooted $d$-ary tree, with the root corresponding to $O_1$; in particular, $\Gamma$ is isomorphic to $\mathcal T_d$. Fixing one such isomorphism, there is an obvious map
\[\tau:\CC_d(O,Y) \to \mathcal T_{d}\] such that the preimage of the root is $O_1$, and the preimage of the midpoint of every edge of $\mathcal T_{d}$ is a suited curve. Next, there is a natural surjective homomorphism 
\[
q:\Map(Y^d, b_0) \to \Sym(d)
\]
which records the manner in which the elements of $\Map(Y^d, b_0)$ permute the boundary components of $Y^d$.

Let $H < \Map(Y^d, b_0)$ and write $\bar H$ for the $q$-image of $H$. Let $f\in \mathfrak B_d(H)$, and choose a tree-pair representative $(Z,(f_Z, \mathbf h), f(Z))$ for $f$. Then $f$ induces a bijection $b_f$ between the leaves of $\tau(Z)$ and $\tau(f(Z))$; also,  on the connected component $X_i$ of $\CC_d(O,Y) \setminus Z$, the homeomorphism $f$ induces, via $\tau$, an isomorphism $f^*_{X_i}:\tau(X_i) \to \tau(f(X_i))$ that acts like $q(h_i)$. As such, the tuple 
\[
\left(\tau(X),(b_f,q(\textbf h)),\tau(f(X)\right)
\]
is an element of $V_{d}(\bar H)$. An exercise in the definitions of  tree-pair representatives shows that the map 

\[
(Z,(f_Z, \mathbf h), f(Z)) \mapsto \left(\tau(X),(b_f,q(\textbf h)),\tau(f(X)\right),
\]
descends to a map 
\[
\pi: \mathfrak B_d(H) \to V_d(\bar H),
\]
which is in addition a homomorphism. Moreover, one has: 
\begin{prop}\label{prop:surj-symthom}
The  homomorphism $\pi: \mathfrak B_d(H) \to V_{d}(\bar H)$ is surjective.
\end{prop}

\begin{proof}
Consider a tree-pair representative $(F_-, (b,\mathbf g), F_+)$ of an element of the group~$V_{d}(\bar H)$, and denote by~$F_-'$ (resp $F'_+$) the result of removing from~$F_-$ (resp. $F_+$) their leaves and the open half edges incident on these leaves. We know that $Z_-:=\tau^{-1}(F'_-)$ and $Z:=\tau^{-1}(F'_+)$ are suited subsurfaces of $\mathcal{C}_d(O,Y)$. 
Let $\hat f:Z_- \to Z_+$ be a homeomorphism such that the induced bijection between the sets of boundary components of $Z_-$ and $Z_+$ coincides with $b$. Writing $X_1, \ldots, X_n$ for the connected components of $\CC_d(O,Y)$, for each $i$ we choose an element $h_i \in H$ with $q(h_i)=g_i$. Then, the element $f\in \mathfrak B_d(H)$ with tree-pair representative  
\[
(Z_-, (\hat f, \mathbf h), Z_+)
\]
maps to $(F_-, (b,\mathbf g), F_+)\in V_{d}(\bar H)$ under $\pi$, as desired. 
\end{proof}

\begin{remark}
After Remarks \ref{rem:morerootsgroups} and \ref{rem:thompson}, a direct adaptation of the arguments above yield the existence of a surjective homomorphism 
\[
\mathfrak B_{d,r}(H) \to V_{d,r}(\bar H),
\]
where $ V_{d,r}(\bar H)$ is the {\em symmetric Higman-Thompson group with $r$ roots.}
\end{remark}

\begin{remark}
Although it is not very relevant to our discussion, it is worth mentioning that the group of block homeomorphisms of $\mathcal C_d(O,Y)$ surjects onto the {\em Neretin group} of \cite{Ne92}. 
\label{rmk:Neretin}
\end{remark}

\section{Finiteness properties}\label{section-fin}

\subsection{The complex}
In this section, we will construct an infinite-dimensional contractible cube complex $\Stein_{d}(O,Y,H)$ on which $\mathfrak B_d(H)$ acts cellularly, closely related to the classical {\em Stein--Farley} complexes for Higman--Thompson groups \cite{Ste92,Far03}, and very much inspired by the complex constructed in \cite{GLU20, SW21a,ABF+21}. 

Here the surfaces $O$ and $Y$ and the integer $d$ will play virtually no role in the discussion, and so we will simply write $\CC=\CC_{d}(O,Y)$  and $\FB=\FB_{d}(H)$. We remark that the definitions and results in this section are direct translations of those in \cite[Section 5]{ABF+21}; we include short proofs for completeness.

\subsection{The Stein complex}
Consider all ordered pairs $(Z, f)$ where $Z$ is a suited subsurface of $\CC$ and $f\in \FB$. We will say a pair $(Z_1, f_1)$ is equivalent to a pair $(Z_2, f_2)$, written $(Z_1, f_1)\sim (Z_2, f_2)$, if and only if there are representing homeomorphisms $f_1$ and $f_2$ (abusing notation) such that $f_2^{-1}\circ f_1$ maps $Z_1$ to $Z_2$ where $Z_1$ is a defining surface for $f_2^{-1}\circ f_1$ as an $H$-block homeomorphism. A straightforward calculation confirms that this relation is symmetric, reflexive, and transitive, i.e. an equivalence relation.

We will denote by $[Z,f]$ the equivalence class containing $(Z,f)$ with respect to this relation, and write $\CS$ for the set of all equivalence classes. Importantly, the group $\FB$ acts on $\CS$ via left multiplication as follows:
\[g\cdot [Z, f]=[Z, g\circ f].\]

Recall that by definition, a suited subsurface is the union of $O_1$ and finitely many additional blocks. We define the \emph{complexity} $h(Z, f)$ of $(Z,f)$ to be the number of blocks in $Z$. We observe that if $(Z_1, f_1)\sim (Z_2, f_2)$ then $h(Z_1,f_1)=h(Z_2, f_2)$ and thus $h$ gives a well-defined function $h: \CS\rightarrow \BN$.

\begin{defn}\label{defn:complexity}
The \emph{complexity} of $[Z,f]\in \CS$ is $h(Z,f)$ for some, and hence any, representative of $[Z,f]$.
\end{defn}

 Let $x_1$ and $x_2$ be two classes in $\CS$ and define $x_1\leq x_2$ if and only if $x_1=[Z_1, f]$ and $x_2=[Z_2, f]$ for suited subsurfaces with $Z_1 \subseteq Z_2$. When the inclusion is proper, we write $x_1 < x_2$. The following is a word-for-word adaptation of Lemmas 5.2 and 5.3 of \cite{ABF+21}: 

\begin{lem}
$(\CS, \leq)$ is a directed poset. Moreover, if $x<y \in \CS$ then $h(x)<h(y)$. In particular, its geometric realization $|\CS|$ is contractible
\end{lem}

Following the same strategy of \cite{ABF+21}, we now consider a second, finer relation $\preceq$ on $\CS$. Let $x_1, x_2$ be two vertices in $\CS$. We will say $x_1\preceq x_2$ if $x_1=[Z, f]$ and $x_2=[Z\cup X, f]$ where $X$ is a disjoint union of blocks. As usual, if $x_1\preceq x_2$ and $x_1\neq x_2$, we will write $x_1 \prec x_2$. This finer order is no longer a partial ordering as it is not transitive. Nevertheless, it is true that if $x_1\preceq x_3$ and $x_1\leq x_2 \leq x_3$ then $x_1\preceq x_2 \preceq x_3$. Thus we can define a simplex $x_0<x_1<\cdots<x_k$ in $|\CS|$ to be \emph{elementary} if $x_0\preceq x_k$.

\begin{defn}[Stein complex] 
The \emph{Stein complex} $\FX=\FX(O,Y,H)$ is the full subcomplex of $|\CS|$ consisting of elementary simplices. 
\end{defn}

We note that the Stein complex $\FX$ is more tractable than $|\CS|$ but still has many of the properties we care about. For instance, the action of $\FB$ on $\CS$ preserves the relation $\preceq$, so thus $\FB$ has a simplicial action on $\FX$. Moreover, a classical result of Quillen \cite{Qui78} implies the following, whose proof is a verbatim copy of Lemmas 5.6 and 5.7 of \cite{ABF+21}: 

\begin{prop}
The Stein complex $\FX$ is contractible.
\end{prop}

\subsection{The Stein-Farley cube complex}
In this section, we  group the simplices of $\FX$ together to form a cube complex. Let $x$ and $y$ be vertices of~$\FX$ with $x\le y$, and set $[x,y]= \{z \mid x\le z\le y\}$. The same argument of \cite[Lemma 5.8]{ABF+21} shows that $[x,y]$ has the structure of a finite Boolean lattice; in particular, its geometric realization is (the barycentric subdivision of) a cube of dimension $h(y) - h(x)$. Moreover, every two cubes intersect over a face, by the same argument as in \cite[Lemma 5.10]{ABF+21}. In particular, we have (see \cite[Lemma~5.11]{ABF+21}):

\begin{prop}
The complex $\FX$ has the structure of a cube complex with each cube defined by an interval $[x,y]$ with $x\preceq y$.
\end{prop}

\begin{defn}[Stein-Farley complex] The complex $\FX$ equipped with the above cubical structure is called the \emph{Stein-Farley cube complex} associated to the Cantor surface $\CC$ and the subgroup $H<\Map_\partial(Y^d, b_0)$.
\end{defn}

\section{Brown's Criterion and Discrete Morse Theory on the Stein-Farley cube complex}
As was the case in \cite{ABF+21}, now that we have an action of $\FB$ on the Stein-Farley complex, we will want to apply Brown's criterion to deduce the finiteness properties \cite{Br87}. First, recall that a filtration $\{K_j\}_{j\geq 1}$ of a space $K$ is called \emph{essentially $n$-connected} if for every $i\geq 1$, there exists an $i'\geq i$ such that $\pi_\ell(K_i\rightarrow K_{i'})$ is trivial for all $\ell \leq n$.

We remark that the definitions and results in this Section are direct translations of those in \cite[Section 6]{ABF+21}; we include short proofs for completeness.

\begin{thm}[Brown's Criterion]
Let $n$ be in $\BN$ and assume that a group~$G$ acts on an $(n-1)$-connected CW-complex $\CK$. Assume that the stabilizer of every $k$-cell of $\CK$ is of type $F_{n-k}$. Let $\{\CK_j\}_{j \geq 1}$ be a filtration of $\CK$ such that each $\CK_j$ is finite modulo the action of $G$. Then $G$ is of type $F_n$ if and only if $\{\CK_j\}_j$ is essentially $(n-1)$-connected.
\end{thm}

First, observe that the complexity function given in Definition~\ref{defn:complexity} gives a function $h:\FX^{(0)}\rightarrow \BN$ which naturally extends affinely to each cube to obtain a \emph{height function}
\[h:\FX\rightarrow \BR,\]
which is $\FB$-equivariant with respect to the trivial action of $\FB$ on $\BR.$ Each cube has a unique vertex at  which $h$ is maximized (respectively minimized) which we call the the \emph{top} (respectively \emph{bottom}) of the cube. Note that this agrees with our definition of top and bottom above. Now, for any integer $k$, we define $\FX^{\leq k}$ to be the subcomplex of $\FX$ spanned by vertices with height at most $k$. An obvious adaptation of \cite[Lemma 6.2]{ABF+21} yields: 

\begin{lem}
The group $\FB$ acts cocompactly on $\FX^{\leq k}$ for all $1\leq k < \infty$.
\end{lem}

Next, we prove:

\begin{lem}\label{lem:stabofcells}
Let $C$ be a cube in $\FX$ with bottom vertex given by $x=[Z, g]$. Then the stabilizer of $C$ in $\FB_d(H)$ is isomorphic to a finite index subgroup of $\Map_\partial(Z) \ltimes H^m$ where $m$ is the number of suited boundary components of~$Z$.
\end{lem}
\begin{proof}
Fix a cube $C$ in $\FX$ and let $x=[Z,g]$ be its bottom vertex. We may assume without loss of generality that $g=\id$. The cube $C$ is spanned by vertices obtained by attaching blocks to a collection~$A$ of  suited boundary components. Now any element $f\in \FB$ stabilizes the cube if and only if it both fixes~$x$ and leaves~$A$ invariant. The element $f$ fixes $x=[Z, \id]$ if and only if $f$ sends $Z$ to $Z$ and is labeled by a tuple $\bf h$ outside of $Z$ and so it is therefore a subgroup of $\Map_\partial(Z) \ltimes H^m$. Moreover $\Map(Z)\ltimes H^m$ preserves the cube and so the result follows from Remark~\ref{rem:finindex}.

\end{proof}

Since mapping class groups of compact surfaces are of type $F_\infty$ \cite{Har71}, and being of type $F_n$  is invariant under short exact sequences and passing to finite-index subgroups \cite[Section 7.2]{Ge08}, the previous lemma implies: 

\begin{cor}
Let $C$ be a cube in $\FX$. Then the stabilizer of $C$ in $\FB_d(H)$ has type $F_n$ if $H$ does. 
\end{cor}

\subsection{Descending links} 
We have now checked all of the conditions of Brown's Criterion except one, namely that the filtration $\{\FX^{\le m}\}_m$ is essentially $(n-1)$-connected. In fact, we will show the connectivity of the pairs $(\FX^{\le m}, \FX^{\le {m+1}})$ tends to $\infty$ as $m$ tends to $\infty$; here, recall that a pair of spaces $(L,K)$ with $K\subseteq L$ is $k$-connected if the inclusion map $K \hookrightarrow L$ induces an isomorphism in their respective homotopy groups $\pi_j$ for $j< k$ and an epimorphism in $\pi_k$. To establish this, we will apply {\em discrete Morse theory} to reduce the problem to analyzing the connectivity of the {\em descending links}.

Let $\CK$ be a piecewise Euclidean cell complex, and let~$h$ be a map from the set of vertices of $\CK$ to the integers, such that each cell has a unique vertex maximizing~$h$.  Call $h$ a \emph{height function}, and $h(y)$ the \emph{height} of $y$ for vertices~$y$ in $\CK$.  For $t\in\BZ$, define $\CK^{\leq t}$ to be the full subcomplex of~$\CK$ spanned by vertices~$y$ with $h(y)\leq t$. Similarly, define $\CK^{< t}$ and  $\CK^{=t}$.  The \emph{descending star} $\dst(y)$ of a vertex $y$ is defined to be the open star of~$y$ in~$\CK^{\le h(y)}$.  The \emph{descending link} $\dlk(y)$ of $y$ is given by the set of ``local directions'' starting at~$y$ and pointing into $\dst(y)$, that is, the link of $y$ in $\dst(y)$. For more details see \cite{BB97}. Now the following Morse Lemma is a consequence of \cite[Corollary~2.6]{BB97}.

\begin{lem}[Morse Lemma]\label{lemm-Morse}
Let $\CK$ be a piecewise Euclidean cell complex and let $h$ be a height function on $\CK$.
\begin{enumerate}
    \item Suppose that for any vertex $y$ with $h(y)=t$, $\dlk (y)$ is $(k-1)$-connected. Then the pair $(\CK^{\leq t}, \CK^{<t})$ is $k$-connected.
  \item Suppose that for  any vertex $y$ with $h(y) \geq t$, $\dlk (y)$ is $(k-1)$-connected. Then $(\CK,\CK^{<t})$ is $k$-connected.
\end{enumerate}
\end{lem}

Recall that the height function for the Stein-Farley complex $\FX$ is the affine extension of the complexity function $h$ given in Definition~\ref{defn:complexity}.

\subsection{From descending links to the complex of labeled blocks}
In this subsection we will determine the connectivity properties of descending links. Behind all of our arguments is Hatcher-Wahl's notion of a {\em complete join complex}, which we now recall:

\begin{defn}[Complete join]
A simplicial map $\pi:K \to L$ between simplicial complexes is said to be a {\em complete join} if 
\begin{itemize}
    \item $\pi$ is surjective;
    \item The restriction of $\pi$ to every simplex of $K$ is injective; 
    \item For each simplex $\sigma=\langle x_0, \ldots, x_p \rangle$ of $L$, the subcomplex of $K$ consisting of those simplices that map to $\sigma$ is equal to the join $\pi^{-1}(x_0) \ast \cdots \ast \pi^{-1}(x_p)$.
\end{itemize}
By a slight abuse of notation, we will say that $K$ is a {\em complete join over} $L$. 
\end{defn}
The interest of the above definition in our situation is the following result, which is Proposition 3.5 of \cite{HW10}; here, recall that a complex is {\em weakly Cohen-Macaulay of dimension $m$} if it is $(m-1)$-connected and the link of every $p$-simplex is $(m-p-2)$-connected. 

\begin{prop}
If $K$ is a complete join complex
over a weakly Cohen-Macaulay  complex $L$ of dimension $m$, then $K$ is also weakly Cohen-Macaulay of dimension $m$.
\end{prop}

Let $x=[X,f]$ be a vertex of $\FX$; we want to analyze the connectivity properties of its descending link. First, we may assume that $f = {\rm id}$. Then a simplex in the descending link of $x$ is determined by $z = [Z,g]$, where there exists a collection $\{Y_0,\ldots, Y_{p}\}$ of pairwise distinct and pairwise disjoint blocks such that $[Z', g]\sim[X, \rm id]$ with $Z'=Z\cup \bigcup_{i=0}^p Y_p$.


Our goal now is to map $\lk^\downarrow(x)$ onto a weakly Cohen-Macaulay complex in such a way that $\lk^\downarrow(x)$ is a complete join complex over it. To this end, we have the following definition:

\begin{defn}[Block Complex]
Given a suited subsurface $W$ and a (not necessarily proper) subset $A$ of boundary components of $W$, the (unlabeled) block complex $\mathcal P(W,A)$ is the simplicial complex whose $p$-simplices are sets $\{Z_0,\ldots, Z_p\}$ of pairwise disjoint isotopy classes of subsurfaces of~$W$, each homeomorphic to $Y^d$, with the property that exactly one boundary component is essential in $W$, while the rest belong to $A$.
\end{defn}

We remark that $\mathcal P(W,A)$ is called the {\em piece complex} in \cite{ABF+21}. The following is  \cite[Theorem 10.1]{ABF+21};

\begin{thm}\label{thm:main-2d}
Let $W_g^b$ be the compact, connected surface of genus $g$ with $b$ boundary components, and let $A$ be a (not necessarily proper) subset of these boundary components. Then $\mathcal P(W_g^b,A)$ is weakly Cohen-Macaulay of dimension $m$ provided that $m \leq \frac{g-1}{2}$ and $m \leq \frac{|A|+2d}{2d-1}$.
\end{thm}


Let $x$ be a vertex of $\FX$; as above, we can assume $x=[X,{\rm id}]$. We now define a map 
$$\Pi: \lk^\downarrow(x) \to \mathcal P(X, A),$$
where $A$ is the set of all suited boundary components of $X$. Let $z\in \lk^\downarrow (x)$. As such, we can write $z=[Z,g]$, where there exists a surface $Z'$ which is obtained from~$Z$ by adding pairwise disjoint pieces $Y_0, \dots, Y_p$ so that~$g$ maps~$Z'$ to $X$ where~$Z'$ is a defining surface for~$g$ as an $H$-block homeomorphism. We then set $\Pi(z)=\{g(Y_0), \dots, g(Y_p)\}$.

We first check that the map $\Pi$ is well-defined; to this end, let $(W, h)$ be another representative of~$z$. As before, this implies there exists a surface~$W'$ obtained from~$W$ by adding pairwise disjoint pieces $Y_0', \dots, Y_p'$ so that~$h$ maps~$W'$ to~$X$. By the definition of the equivalence relation, this means $h^{-1}\circ g$ maps~$Z$ to~$W$ and in particular, $g(Z)=h(W)\subseteq X$. Therefore, $\{g(Y_0), \dots, g(Y_p)\}=\{h(Y_0'), \dots, h(Y_p')\}$, and we are done.

Next, we have:

\begin{prop}\label{prop:join}
The map $\Pi$ is a complete join.
\end{prop}

\begin{remark}\label{rem:chris}
In a previous version, we claimed that there was an isomorphism $\Pi$ from the descending link to the \emph{labeled block complex}. We are grateful to Chris Leininger for pointing out that this was not correct, and for suggesting the strategy for showing that it is a complete join. 
\end{remark}

\begin{proof}[Proof of Proposition \ref{prop:join}]

Write $x=[X, \rm id]$, and let $\{Z_0,\ldots, Z_p\}$ be a $p$-simplex in $\mathcal P(X, A)$. Let $\{Y_0, \ldots, Y_p\}$ be a collection of disjoint blocks in $X$ such that each $Y_i$ has exactly two non-essential boundary components. 
Now,  $X\setminus \cup_{i=0}^p Z_i$ is homeomorphic to $X\setminus \cup_{i=0}^p Y_i$ and thus there is a homeomorphism $g:X\rightarrow X$ which takes $X\setminus \cup_{i=0}^p Y_i$ to $X\setminus \cup_{i=0}^p Z_i$ and such that $g(Y_i)=Z_i$. We conclude that the map $\Pi$ is surjective.

Note that the disjointness condition on the links and on the block complex and the fact that this property is naturally preserved by $\Pi$ implies that the map is simplexwise injective.

It remains to show that the pre-image of a simplex is the join of the pre-images of its vertices. Clearly the pre-image of a simplex is contained in the join of the pre-images of its vertices so we only need to show the reverse inclusion. To this end, let $\sigma= \{Z_0,\ldots, Z_p\}$ be an $p$-simplex in $\mathcal P(X, A)$. Take vertices \[[W_0, g_0], \ldots, [W_p,g_p]\] of $\mathfrak X$  such that $\Pi([W_i,g_i])=Z_i$; we want to see that these vertices in fact span an $p$-simplex in $\lk^\downarrow(x)$.

We now make several reductions. We first claim that we may assume that, for all $i=0, \ldots, p$, we have that $W_i\cup g_i^{-1}(Z_i)=X$, so that $g_i^{-1}(Z_i)$ is a block and $g_i:X\rightarrow X$ is an element of $\FB$ with defining surface $X$. To see this, let $Y_i=g_i^{-1}(Z_i)$, which by construction is a block, and consider any $H$-block homeomorphism $h$ which takes $W_i\cup Y_i$ to $X$ and such that $W_i$ is a defining surface for $h$, i.e. the image of $Y_i$ is also a block. Now we see that $(W_i, g_i)\sim (h(W_i), g_i \circ h^{-1})$, and that this new representative has the desired properties. 

We thus assume without loss of generality that $W_i\cup Y_i=X$ and $g_i(X)=X$ for all $i=\ldots, p$. Next, we claim that we can assume that the non-essential boundary components of $Y_i=g_i^{-1}(Z_i)$ are disjoint from those of $Y_j=g_j^{-1}(Z_j)$ for all $i\neq j$. Indeed, given any block $Y$ with two non-essential boundary components, there is an $H$-block homeomorphism $f$ with $f(X)=X$ and $f(Y_i)=Y$. Moreover $(W_i, g_i)\sim (f(W_i), g_i\circ f^{-1})$. Since $f(Y_i)=Y$, we can freely choose the location of the block $Y_i$ in $X$ among blocks with two non-essential boundary components. 

We thus assume without loss of generality that $Y_i$ is disjoint from $Y_j$ for each $i\neq j$. Finally,  note that since $Y_0, \dots, Y_p$ are pairwise disjoint and $g_i(Y_i)$ are pairwise disjoint and the complements $X\backslash Y_i$ are homeomorphic to the complements $X\backslash g_i(Y_i)$, we can construct an element $g$ in $\FB$ with $g(X)=X$ and which agrees with $g_i$ on $W_i$. In particular, the collection $\{[W_i, g]\}_{i=0}^{i=p}$ spans a simplex in $\lk^\downarrow(x)$.


\end{proof}

\begin{cor}[Part ``if" of Theorem \ref{thm:main}]
If $H$ is of type $F_n$, so is $\mathfrak B_d(H)$.
\end{cor}

\begin{proof}
By the above results, since $H$ is of type $F_n$ then we can apply Brown's criterion above to the action of $\mathfrak B_d(H)$ on the cube complex $\mathfrak X_d(O,Y,H)$, thus obtaining that $\mathfrak B_d(H)$ is of type $F_n$, as desired. 
\end{proof}

\section{From global to local}
Finally, in this section we prove that if $\mathfrak B_{d}(H)$ is of type $F_n$, then so is~$H$.  In a nutshell, we will show that $H$ is a {\em quasi-retract} of $\mathfrak B_{d}(H)$, following the strategy in \cite[Section 3.5]{SW21a}. 

Before we proceed, we briefly recall the definition of  quasi-retraction between metric spaces. 
Let $X$ and $Y$ are metric spaces. A function $f:X\to Y$ is said to be {\em coarsely Lipschitz} if there exist constants $C,D>0$ so that 
$$d(f(x),f(x')) \leq Cd(x,x')+D \text{ for all } x,x'\in X.$$

A coarsely Lipschitz function $\rho:X\to Y$ is said to be a quasi-retraction if there exists a coarse Lipschitz function $\iota:Y\to X$ and a constant $E>0$ so that $d(\rho\circ \iota(y),y)\leq E$ for all $y\in Y$. In this situation, we say that $Y$ is a {\em quasi-retract} of $X$. In the realm of groups, we say that a finitely generated group $Q$ is a quasi-retract of another finitely generated group $G$ if $Q$ is a quasi-retract of $G$ with respect to the word metric given by some (and hence, any) choice of finite generating sets. We have:

\begin{thm}\cite[Theorem 8]{Al94}\label{thm-quasi-re-fin}
Let $G$ and $Q$ be finitely generated groups such that $Q$ is a quasi-retract of $G$. If $G$ is of type $F_n$, so is $Q$. 
\end{thm}

After these definitions, observe that we have a canonical inclusion homomorphism
\[\iota: H \rightarrow \mathfrak B_{d}(H)\] given by the rule $h\mapsto [O_2, (\id, (h, 1, \ldots, 1)), O_2]$, where recall that $O_2$ is the surface that appears at the second step in the construction of $\CC_d(O,Y)$; in other words, $\iota(H)$ labels the leftmost suited curve of $O_2$ by $h$ and the other suited curves by the identity. Next, we define a map 
\[r:\mathfrak B_d(H) \to H\] by setting $r(f) = h_1$, where $f=[Z,(f_Z, \mathbf h), f(Z)]$ and $\mathbf h=(h_1, \ldots, h_n)$; observe that $r(f)$ does not depend on the particular choice of such representative. We have: 

\begin{prop}\label{prop-fg}
$\mathfrak B_{d}(H)$ is generated by $\iota(H)$ and $\mathfrak B_{d} = \mathfrak B_{d}(\{1\})$. 
\end{prop}

\begin{proof}
Let $G$ be the group generated by $\iota(H)$ and $\mathfrak B_d$; we want to prove that $G=\mathfrak B_d(H)$. To this end, let $f\in \mathfrak B_{d}(H)$ and choose a tree-pair representative $(Z, (f_Z, \mathbf h), f(Z))$ for it, observing that 
\[[Z, (f_Z, \mathbf h), f(Z)] = [Z, (f_Z, {\rm id}), f(Z)]\cdot[f(Z),({\rm id}, \mathbf h), f(Z)]\]
Since $[Z, (f_Z, {\rm id}), f(Z)] \in \mathfrak B_d$, it suffices to show that an element of the form $[X,({\rm id}, \mathbf h), X]$ belongs to $G$. We can assume that $O_2\subseteq X$.

Suppose for the time being that $X=O_2$. In this case, let\\
$\lambda_j(h) := (1,\cdots,1,h,1\cdots,1)$ where the only nontrivial labeling $h$ appears on the $j$-th position, then 
\[
[O_2, ({\rm id}, \lambda_j(h)), O_2]= [O_2, (g, {\rm id}), O_2][O_2, ({\rm id}, \lambda_1(h)), O_2][O_2, (g^{-1}, {\rm id}), O_2],
\] 
where $g \in \mathfrak B_d$ is any element supported on $O_2$ and that interchanges the first and $j$-th suited boundary components of $O_2$, fixing the rest. Since \[[O_2, ({\rm id}, \lambda_1(h)), O_2] \in \iota(H)\] and for $\mathbf h=(h_1\cdots,h_d)$\[
[O_2, ({\rm id},  \mathbf h), O_2] = \prod_{j=1}^d [O_2, ({\rm id},  \lambda_j(h_j)), O_2],
\] we are done in this case.

Suppose then that $X$ is a suited subsurface whose left most suited loop is also a suited loop of $O_2$.  By the same conjugation argument as in the previous paragraph, we can assume that $\mathbf h=(h_1, 1, \ldots, 1)$. But then  using the equivalence relation, we have
\[
[X,({\rm id}, (h_1, 1, \ldots, 1)), X] = [O_2,({\rm id}, (h_1, 1, \ldots, 1)), O_2],
\]
which already lies in the subgroup generated by $\iota(H)$ and $\mathfrak B_{d}$. Now for an arbitrary suited subsurface $X$, again using the conjugation argument, we can assume that $\mathbf h=(h_1, 1, \ldots, 1)$. But then $[X,({\rm id}, \mathbf h_1), X]$ is conjugate to $[X,(\mathrm{id},(1,\cdots,h_1),X]$ via an element in $\mathfrak B_{d}$. Now we can find a suitable smaller suited subsurface $X'$ whose left most suited loop is also a suited loop of $O_2$, such that  \[[X,(\mathrm {id},(1,\cdots,h_1)),X] = [X',(\mathrm {id},(1,\cdots,h_1)),X'].\]
This is known to lie in the subgroup generated by $\iota(H)$ and $\mathfrak B_{d}$ by the previous step and hence we are done. 
\end{proof}

\begin{thm}\label{thm:qs-retract-bmap}
When $H$ is finitely generated, it is a quasi-retract of $\mathfrak B_d(H)$.
\end{thm}

\begin{proof}
Fix  finite generating sets $S_H$ and $S_B$ for $H$ and $\mathfrak B_d$, respectively. By Proposition \ref{prop-fg}, $\iota(S_H)\cup S_B$ is a finite generating set of $\mathfrak B_d(H)$. We will show that the map $r:  \mathfrak B_d(H) \rightarrow H$ described above is coarsely Lipschitz with respect to the word metrics on $\mathfrak B_d(H)$ and $H$. To this end: 
\begin{enumerate}
    \item  $r( g\iota(s))\in \{r(g), r(g)s\}$ for all $s\in S_H$ and all $g\in \mathfrak B_d(H)$;
    \item  $r(gg') = r(g)$ for any $g'\in \mathfrak B_d$ and $g\in \mathfrak B_d(H)$.
\end{enumerate}
In particular, $r$ is nonexpanding and hence coarsely Lipschitz. Since $\iota$ is a group homomorphism, it must be coarsely Lipschitz as well. As $r\circ \iota = \id_H$, we conclude that $r$ is a quasi-retraction. 
\end{proof}

\begin{cor}[Part ``only if" of Theorem~\ref{thm:main}]
If $\mathfrak B_d(H)$ is of type $F_n$, so is $H$.
\end{cor}
\proof 

Let us first prove the corollary for $n=1$. If $H$ is not finitely generated, we can find an increasing  family   of  subgroups $H_1\lneq H_2\lneq H_n\lneq \cdots$ of~$H$ such that $\cup_{n=1}^{\infty}H_i =H$. But then we would have a family of subgroups $\mathfrak B_d(H_1)\lneq \mathfrak B_d(H_2)\lneq \mathfrak B_d(H_n)\lneq \cdots$ of $\mathfrak B_d(H)$, such that $\cup_{n=1}^{\infty}\mathfrak B_d(H_i) =\mathfrak B_d(H)$. This implies $\mathfrak B_d(H)$ is not finitely generated ,which is a contradiction. 

Suppose now $\mathfrak B_d(H)$ is of type $F_n$ for some $n\geq 2$. As above, $H$ is finitely generated, and the corollary now follows from Theorem~\ref{thm-quasi-re-fin} and Theorem~\ref{thm:qs-retract-bmap}. 

\qed

\printbibliography

@article{ABF+21,
    title={Generalized asymptotic mapping class groups and their finiteness properties},
    author={Aramayona, Javier and Bux, Kai-Uwe and Fleschig, Jonas and Petrosyan, Nansen and Wu, Xiaolei},
    year={2021},
    Note={Preprint. arxiv: 2110.05318},
    primaryClass={math.GR}
}

@article {AF21,
    AUTHOR = {Aramayona, Javier and Funar, Louis},
     TITLE = {Asymptotic mapping class groups of closed surfaces punctured
              along {C}antor sets},
   JOURNAL = {Mosc. Math. J.},
  FJOURNAL = {Moscow Mathematical Journal},
    VOLUME = {21},
      YEAR = {2021},
    NUMBER = {1},
     PAGES = {1--29},
      ISSN = {1609-3321},
   MRCLASS = {57K20 (20F65)},
  MRNUMBER = {4219034},
       DOI = {10.17323/1609-4514-2021-21-1-1-29},
       URL = {https://doi-org.proxy.lib.ohio-state.edu/10.17323/1609-4514-2021-21-1-1-29},
}

@incollection {AV20,
    AUTHOR = {Aramayona, Javier and Vlamis, Nicholas G.},
     TITLE = {Big mapping class groups: an overview},
 BOOKTITLE = {In the tradition of {T}hurston---geometry and topology},
     PAGES = {459--496},
 PUBLISHER = {Springer, Cham},
      YEAR = {2020},
   MRCLASS = {57K20},
  MRNUMBER = {4264585},
       DOI = {10.1007/978-3-030-55928-1\_12},
       URL = {https://doi-org.proxy.lib.ohio-state.edu/10.1007/978-3-030-55928-1_12},
}

@article {Al94,
    AUTHOR = {Alonso, Juan M.},
     TITLE = {Finiteness conditions on groups and quasi-isometries},
   JOURNAL = {J. Pure Appl. Algebra},
  FJOURNAL = {Journal of Pure and Applied Algebra},
    VOLUME = {95},
      YEAR = {1994},
    NUMBER = {2},
     PAGES = {121--129},
      ISSN = {0022-4049},
   MRCLASS = {20J99 (20F32)},
  MRNUMBER = {1293049},
MRREVIEWER = {Thomas Brady},
       DOI = {10.1016/0022-4049(94)90069-8},
       URL = {https://doi-org.proxy.lib.ohio-state.edu/10.1016/0022-4049(94)90069-8},
}

@article {Ne92,
    AUTHOR = {Neretin, Yu. A.},
     TITLE = {Combinatorial analogues of the group of diffeomorphisms of the
              circle},
   JOURNAL = {Izv. Ross. Akad. Nauk Ser. Mat.},
  FJOURNAL = {Izvestiya Rossiiskoi Akademii Nauk. Seriya Matematicheskaya},
    VOLUME = {56},
      YEAR = {1992},
    NUMBER = {5},
     PAGES = {1072--1085},
      ISSN = {1607-0046},
   MRCLASS = {22E65},
  MRNUMBER = {1209033},
       DOI = {10.1070/IM1993v041n02ABEH002264},
       URL = {https://doi.org/10.1070/IM1993v041n02ABEH002264},
}

@article {BB97,
    AUTHOR = {Bestvina, Mladen and Brady, Noel},
     TITLE = {Morse theory and finiteness properties of groups},
   JOURNAL = {Invent. Math.},
  FJOURNAL = {Inventiones Mathematicae},
    VOLUME = {129},
      YEAR = {1997},
    NUMBER = {3},
     PAGES = {445--470},
      ISSN = {0020-9910},
   MRCLASS = {20F36 (20J05 57M07)},
  MRNUMBER = {1465330},
MRREVIEWER = {John Meier},
       DOI = {10.1007/s002220050168},
       URL = {https://doi-org.proxy.lib.ohio-state.edu/10.1007/s002220050168}
}

@article {Bri07,
    AUTHOR = {Brin, Matthew G.},
     TITLE = {The algebra of strand splitting. {I}. {A} braided version of
              {T}hompson's group {$V$}},
   JOURNAL = {J. Group Theory},
  FJOURNAL = {Journal of Group Theory},
    VOLUME = {10},
      YEAR = {2007},
    NUMBER = {6},
     PAGES = {757--788},
      ISSN = {1433-5883},
   MRCLASS = {20F65 (20F05 20F36)},
  MRNUMBER = {2364825},
MRREVIEWER = {Zoran \v{S}uni\'{c}},
       DOI = {10.1515/JGT.2007.055},
       URL = {https://doi-org.proxy.lib.ohio-state.edu/10.1515/JGT.2007.055},
}

@inproceedings {Br87,
    AUTHOR = {Brown, Kenneth S.},
     TITLE = {Finiteness properties of groups},
 BOOKTITLE = {Proceedings of the {N}orthwestern conference on cohomology of
              groups ({E}vanston, {I}ll., 1985)},
   JOURNAL = {J. Pure Appl. Algebra},
  FJOURNAL = {Journal of Pure and Applied Algebra},
    VOLUME = {44},
      YEAR = {1987},
    NUMBER = {1-3},
     PAGES = {45--75},
      ISSN = {0022-4049},
   MRCLASS = {20J05 (11F75 20F05 22E40)},
  MRNUMBER = {885095},
MRREVIEWER = {Ralph Strebel},
       DOI = {10.1016/0022-4049(87)90015-6},
       URL = {https://doi-org.proxy.lib.ohio-state.edu/10.1016/0022-4049(87)90015-6},
}

@article {BG84,
    AUTHOR = {Brown, Kenneth S. and Geoghegan, Ross},
     TITLE = {An infinite-dimensional torsion-free {${\rm FP}_{\infty }$}
              group},
   JOURNAL = {Invent. Math.},
  FJOURNAL = {Inventiones Mathematicae},
    VOLUME = {77},
      YEAR = {1984},
    NUMBER = {2},
     PAGES = {367--381},
      ISSN = {0020-9910},
   MRCLASS = {20J05 (55P99)},
  MRNUMBER = {752825},
MRREVIEWER = {G. Peter Scott},
       DOI = {10.1007/BF01388451},
       URL = {https://doi-org.proxy.lib.ohio-state.edu/10.1007/BF01388451},
}

@article {BFM+16,
    AUTHOR = {Bux, Kai-Uwe and Fluch, Martin G. and Marschler, Marco and
              Witzel, Stefan and Zaremsky, Matthew C. B.},
     TITLE = {The braided {T}hompson's groups are of type {$\rm F_\infty$}},
      NOTE = {With an appendix by Zaremsky},
   JOURNAL = {J. Reine Angew. Math.},
  FJOURNAL = {Journal f\"{u}r die Reine und Angewandte Mathematik. [Crelle's
              Journal]},
    VOLUME = {718},
      YEAR = {2016},
     PAGES = {59--101},
      ISSN = {0075-4102},
   MRCLASS = {20F65 (20F36 20J05 57M07)},
  MRNUMBER = {3545879},
MRREVIEWER = {Ian J. Leary},
       DOI = {10.1515/crelle-2014-0030},
       URL = {https://doi-org.proxy.lib.ohio-state.edu/10.1515/crelle-2014-0030},
}

@article {Deh06,
    AUTHOR = {Dehornoy, Patrick},
     TITLE = {The group of parenthesized braids},
   JOURNAL = {Adv. Math.},
  FJOURNAL = {Advances in Mathematics},
    VOLUME = {205},
      YEAR = {2006},
    NUMBER = {2},
     PAGES = {354--409},
      ISSN = {0001-8708},
   MRCLASS = {20F36 (57M25 57S05)},
  MRNUMBER = {2258261},
MRREVIEWER = {Darren D. Long},
       DOI = {10.1016/j.aim.2005.07.012},
       URL = {https://doi-org.proxy.lib.ohio-state.edu/10.1016/j.aim.2005.07.012},
}

@book {FM12,
    AUTHOR = {Farb, Benson and Margalit, Dan},
     TITLE = {A primer on mapping class groups},
    SERIES = {Princeton Mathematical Series},
    VOLUME = {49},
 PUBLISHER = {Princeton University Press, Princeton, NJ},
      YEAR = {2012},
     PAGES = {xiv+472},
      ISBN = {978-0-691-14794-9},
   MRCLASS = {57M50 (20F36 20F65 57M07 57N05)},
  MRNUMBER = {2850125},
MRREVIEWER = {Stephen P. Humphries},
}

@article {Far03,
    AUTHOR = {Farley, Daniel S.},
     TITLE = {Finiteness and {$\rm CAT(0)$} properties of diagram groups},
   JOURNAL = {Topology},
  FJOURNAL = {Topology. An International Journal of Mathematics},
    VOLUME = {42},
      YEAR = {2003},
    NUMBER = {5},
     PAGES = {1065--1082},
      ISSN = {0040-9383},
   MRCLASS = {20F65 (20F06 20F67 57M07)},
  MRNUMBER = {1978047},
MRREVIEWER = {Hanspeter Fischer},
       DOI = {10.1016/S0040-9383(02)00029-0},
       URL = {https://doi-org.proxy.lib.ohio-state.edu/10.1016/S0040-9383(02)00029-0},
}

@article {FH15,
    AUTHOR = {Farley, Daniel S. and Hughes, Bruce},
     TITLE = {Finiteness properties of some groups of local similarities},
   JOURNAL = {Proc. Edinb. Math. Soc. (2)},
  FJOURNAL = {Proceedings of the Edinburgh Mathematical Society. Series II},
    VOLUME = {58},
      YEAR = {2015},
    NUMBER = {2},
     PAGES = {379--402},
      ISSN = {0013-0915},
   MRCLASS = {20F65},
  MRNUMBER = {3341445},
MRREVIEWER = {Ariadna Fossas},
       DOI = {10.1017/S001309151400011X},
       URL = {https://doi-org.proxy.lib.ohio-state.edu/10.1017/S001309151400011X},
}

@article {FMW+13,
    AUTHOR = {Fluch, Martin G. and Marschler, Marco and Witzel, Stefan and
              Zaremsky, Matthew C. B.},
     TITLE = {The {B}rin-{T}hompson groups {$sV$} are of type
              {$\text{F}_\infty$}},
   JOURNAL = {Pacific J. Math.},
  FJOURNAL = {Pacific Journal of Mathematics},
    VOLUME = {266},
      YEAR = {2013},
    NUMBER = {2},
     PAGES = {283--295},
      ISSN = {0030-8730},
   MRCLASS = {20F65 (05E18)},
  MRNUMBER = {3130623},
MRREVIEWER = {Shengkui Ye},
       DOI = {10.2140/pjm.2013.266.283},
       URL = {https://doi-org.proxy.lib.ohio-state.edu/10.2140/pjm.2013.266.283},
}

@article {FK04,
    AUTHOR = {Funar, L. and Kapoudjian, C.},
     TITLE = {On a universal mapping class group of genus zero},
   JOURNAL = {Geom. Funct. Anal.},
  FJOURNAL = {Geometric and Functional Analysis},
    VOLUME = {14},
      YEAR = {2004},
    NUMBER = {5},
     PAGES = {965--1012},
      ISSN = {1016-443X},
   MRCLASS = {57M50 (20F65 57N05)},
  MRNUMBER = {2105950},
       DOI = {10.1007/s00039-004-0480-9},
       URL = {https://doi-org.proxy.lib.ohio-state.edu/10.1007/s00039-004-0480-9},
}

@article {FK08,
    AUTHOR = {Funar, Louis and Kapoudjian, Christophe},
     TITLE = {The braided {P}tolemy-{T}hompson group is finitely presented},
   JOURNAL = {Geom. Topol.},
  FJOURNAL = {Geometry \& Topology},
    VOLUME = {12},
      YEAR = {2008},
    NUMBER = {1},
     PAGES = {475--530},
      ISSN = {1465-3060},
   MRCLASS = {20F36 (20F05 20F10 57M07 57M25)},
  MRNUMBER = {2390352},
MRREVIEWER = {Alexey Muranov},
       DOI = {10.2140/gt.2008.12.475},
       URL = {https://doi-org.proxy.lib.ohio-state.edu/10.2140/gt.2008.12.475},
}

@article {FK09,
    AUTHOR = {Funar, Louis and Kapoudjian, Christophe},
     TITLE = {An infinite genus mapping class group and stable cohomology},
   JOURNAL = {Comm. Math. Phys.},
  FJOURNAL = {Communications in Mathematical Physics},
    VOLUME = {287},
      YEAR = {2009},
    NUMBER = {3},
     PAGES = {784--804},
      ISSN = {0010-3616},
   MRCLASS = {57N05 (20F05 20F38 57R56)},
  MRNUMBER = {2486661},
MRREVIEWER = {Athanase Papadopoulos},
       DOI = {10.1007/s00220-009-0728-1},
       URL = {https://doi-org.proxy.lib.ohio-state.edu/10.1007/s00220-009-0728-1},
}

@article {FN18,
    AUTHOR = {Funar, Louis and Neretin, Yurii},
     TITLE = {Diffeomorphism groups of tame {C}antor sets and
              {T}hompson-like groups},
   JOURNAL = {Compos. Math.},
  FJOURNAL = {Compositio Mathematica},
    VOLUME = {154},
      YEAR = {2018},
    NUMBER = {5},
     PAGES = {1066--1110},
      ISSN = {0010-437X},
   MRCLASS = {57M50 (20F36 37C85 54H15 57S05)},
  MRNUMBER = {3798595},
MRREVIEWER = {Matthew C. B. Zaremsky},
       DOI = {10.1112/S0010437X18007066},
       URL = {https://doi-org.proxy.lib.ohio-state.edu/10.1112/S0010437X18007066},
}

@article{GLU20,
    title={Asymptotically rigid mapping class groups I: Finiteness properties of braided {T}hompson's and {H}oughton's groups},
    author={Genevois, Anthony and Lonjou, Anne and Urech, Christian},
    year={2020},
    journal={Geom. Top.},
    Note={To appear. arxiv:2010.07225}
    }

@book {Ge08,
    AUTHOR = {Geoghegan, Ross},
     TITLE = {Topological methods in group theory},
    SERIES = {Graduate Texts in Mathematics},
    VOLUME = {243},
 PUBLISHER = {Springer, New York},
      YEAR = {2008},
     PAGES = {xiv+473},
      ISBN = {978-0-387-74611-1},
   MRCLASS = {57M07 (20F65 20J05 55-02 55P57 57-02)},
  MRNUMBER = {2365352},
MRREVIEWER = {John G. Ratcliffe},
       DOI = {10.1007/978-0-387-74614-2},
       URL = {https://doi-org.proxy.lib.ohio-state.edu/10.1007/978-0-387-74614-2},
}

@article {Har71,
    AUTHOR = {Harvey, W. J.},
     TITLE = {On branch loci in {T}eichm\"{u}ller space},
   JOURNAL = {Trans. Amer. Math. Soc.},
  FJOURNAL = {Transactions of the American Mathematical Society},
    VOLUME = {153},
      YEAR = {1971},
     PAGES = {387--399},
      ISSN = {0002-9947},
   MRCLASS = {30A60 (58D15)},
  MRNUMBER = {297994},
MRREVIEWER = {L. Keen},
       DOI = {10.2307/1995564},
       URL = {https://doi-org.proxy.lib.ohio-state.edu/10.2307/1995564},
}

@article {HW10,
    AUTHOR = {Hatcher, Allen and Wahl, Nathalie},
     TITLE = {Stabilization for mapping class groups of 3-manifolds},
   JOURNAL = {Duke Math. J.},
  FJOURNAL = {Duke Mathematical Journal},
    VOLUME = {155},
      YEAR = {2010},
    NUMBER = {2},
     PAGES = {205--269},
      ISSN = {0012-7094},
   MRCLASS = {57M07 (20F28)},
  MRNUMBER = {2736166},
MRREVIEWER = {Mihalis A. Sykiotis},
       DOI = {10.1215/00127094-2010-055},
       URL = {https://doi-org.proxy.lib.ohio-state.edu/10.1215/00127094-2010-055},
}

@book {Hi74,
    AUTHOR = {Higman, Graham},
     TITLE = {Finitely presented infinite simple groups},
    SERIES = {Notes on Pure Mathematics, No. 8},
 PUBLISHER = {Australian National University, Department of Pure
              Mathematics, Department of Mathematics, I.A.S., Canberra},
      YEAR = {1974},
     PAGES = {vii+82},
   MRCLASS = {20F05 (20B25)},
  MRNUMBER = {0376874},
MRREVIEWER = {F. Levin},
}

@book{Ke23,
  title     = {Vorlesungen \"{u}ber Topologie. I},
  author    = {Ker\'{e}kj\'{a}rt\'{o}, B. v.},
  year      = {1923},
  publisher = {Spring},
  address   = {Berlin}
}

@article {Ri63,
    AUTHOR = {Richards, Ian},
     TITLE = {On the classification of noncompact surfaces},
   JOURNAL = {Trans. Amer. Math. Soc.},
  FJOURNAL = {Transactions of the American Mathematical Society},
    VOLUME = {106},
      YEAR = {1963},
     PAGES = {259--269},
      ISSN = {0002-9947},
   MRCLASS = {54.75},
  MRNUMBER = {143186},
MRREVIEWER = {S. S. Cairns},
       DOI = {10.2307/1993768},
       URL = {https://doi-org.proxy.lib.ohio-state.edu/10.2307/1993768},
}

@article {Qui78,
    AUTHOR = {Quillen, Daniel},
     TITLE = {Homotopy properties of the poset of nontrivial {$p$}-subgroups
              of a group},
   JOURNAL = {Adv. in Math.},
  FJOURNAL = {Advances in Mathematics},
    VOLUME = {28},
      YEAR = {1978},
    NUMBER = {2},
     PAGES = {101--128},
      ISSN = {0001-8708},
   MRCLASS = {20J99},
  MRNUMBER = {493916},
MRREVIEWER = {Kenneth S. Brown},
       DOI = {10.1016/0001-8708(78)90058-0},
       URL = {https://doi-org.proxy.lib.ohio-state.edu/10.1016/0001-8708(78)90058-0},
}

@article {Ste92,
    AUTHOR = {Stein, Melanie},
     TITLE = {Groups of piecewise linear homeomorphisms},
   JOURNAL = {Trans. Amer. Math. Soc.},
  FJOURNAL = {Transactions of the American Mathematical Society},
    VOLUME = {332},
      YEAR = {1992},
    NUMBER = {2},
     PAGES = {477--514},
      ISSN = {0002-9947},
   MRCLASS = {20F32 (20E32 20J05 55P20 57M05 57Q99)},
  MRNUMBER = {1094555},
MRREVIEWER = {Ross Geoghegan},
       DOI = {10.2307/2154179},
       URL = {https://doi-org.proxy.lib.ohio-state.edu/10.2307/2154179},
}

@article{SW21a,
    title={Finiteness properties for relatives of braided {H}igman--{T}hompson groups},
    author={Skipper, Rachel and Wu, Xiaolei},
    year={2021},
    Note={Preprint. arxiv:2103.14589v2},
    primaryClass={math.GR}
}

@article{SW21b,
    title={Homological stability for the ribbon {H}igman--{T}hompson groups},
    author={Skipper, Rachel and Wu, Xiaolei},
    year={2021},
    Note={Preprint. arxiv:2106.08751},
    primaryClass={math.GR}
}

@article {Za17,
    AUTHOR = {Zaremsky, Matthew C. B.},
     TITLE = {Separation in the {BNSR}-invariants of the pure braid groups},
   JOURNAL = {Publ. Mat.},
  FJOURNAL = {Publicacions Matem\`atiques},
    VOLUME = {61},
      YEAR = {2017},
    NUMBER = {2},
     PAGES = {337--362},
      ISSN = {0214-1493},
   MRCLASS = {20F65 (20F36 57M07)},
  MRNUMBER = {3677865},
MRREVIEWER = {Ioannis Diamantis},
       DOI = {10.5565/PUBLMAT6121702},
       URL = {https://doi.org/10.5565/PUBLMAT6121702},
}

\end{document}